\documentclass[12pt,reqno]{amsart}
\usepackage{amsmath, amssymb, amsfonts, amsthm}
\usepackage{mathtools,xcolor}
\allowdisplaybreaks
\usepackage{hyperref}
\usepackage[all,cmtip]{xy}
\usepackage{tkz-euclide}
\usepackage{tikz-cd}
\usepackage{extarrows}
\usepackage[most]{tcolorbox}   
\usepackage{xcolor}            
\newtheorem{thm}{Theorem}[section]
\newtheorem*{thm*}{Theorem}
\newtheorem{lem}[thm]{Lemma}
\newtheorem{prop}[thm]{Proposition}
\theoremstyle{definition}
\newtheorem{dfn}[thm]{Definition}

\theoremstyle{plain}

\numberwithin{equation}{section}
\usepackage[a4paper, top = 1.2in, bottom = 1.2in, left = 1.2in, right = 1.2in]{geometry}
\numberwithin{equation}{section}

\newcommand{\N}{\mathbb{N}}
\newcommand{\Q}{\mathbb{Q}}

\newcommand{\Z}{\mathbb{Z}}

\newcommand{\F}{\mathbb{F}}

\newcommand{\mcS}{\mathcal{S}}

\newcommand{\mcW}{\mathcal{W}}

\newcommand{\mfa}{\mathfrak{a}}

\newcommand{\mfp}{\mathfrak{p}}
\newcommand{\mfl}{\mathfrak{l}}

\newcommand{\SL}{\mathrm{SL}}
\newcommand{\GL}{\mathrm{GL}}
\newcommand{\GU}{\mathrm{GU}}
\newcommand{\GO}{\mathrm{GO}}

\newcommand{\PSL}{\mathrm{PSL}}

\newcommand{\Aut}{\mathrm{Aut}}
\newcommand{\End}{\mathrm{End}}

\newcommand{\Spec}{\mrm{Spec}}
\newcommand{\sep}{\mrm{sep}}

\newcommand{\Gal}{\mrm{Gal}}
\newcommand{\Ima}{\mrm{Im }}

\newcommand{\un}{\mrm{un}}

\newcommand{\Sp}{\textrm{Sp}}

\newcommand{\lra}{\longrightarrow}
\newcommand{\ra}{\rightarrow}

\newcommand{\mrm}[1]{\mathrm{#1}}

\def\1{1\!\!1}

\newcommand{\psmat}[4]{\bigl( \begin{smallmatrix} #1 & #2 \\ #3 & #4 \end{smallmatrix} \bigr)}

\newcommand{\csmat}[9]{\biggl( \begin{smallmatrix} #1 & #2 & #3 \\
		#4 & #5 & #6 \\ #7 & #8 & #9 \end{smallmatrix} \biggr)}

\title[Surjectivity, Density of Drinfeld modules of rank $2$ and $3$]{On the surjectivity of $(T)$-adic Galois Representations of Drinfeld $A$-Modules of Rank $2$ and $3$: Density results}
\author[N. Kumar]{Narasimha Kumar}
\email{narasimha@math.iith.ac.in}
\address{
	Department of Mathematics \\
	Indian Institute of Technology Hyderabad\\
	Kandi, Sangareddy - 502284\\
	INDIA.
}
\author[D. Shit]{Dwipanjana Shit}
\email{ma22resch01001@iith.ac.in}
\address{
	Department of Mathematics \\
	Indian Institute of Technology Hyderabad\\
	Kandi, Sangareddy - 502284\\
	INDIA.
}

\begin{document}
	\begin{abstract}
		Let $\F_{q}$ be a finite field,  and  $A:=\F_{q}[T]$. In this article, we give explicit criteria, involving concrete valuations, on the coefficients of the Drinfeld $A$-modules of rank $ r$ for $r=2,3$, which ensure the surjectivity of the associated $(T)$-adic Galois representation. As a result, we shall calculate the densities of such Drinfeld $A$-modules. 
    
	\end{abstract}
	\maketitle

	\section{Introduction and Literature}
	
	\subsection{Introduction}
    In his seminal work~\cite{Ser72}, Serre studied the adelic Galois representations attached to elliptic curves over $\Q$ without complex multiplication (CM) and proved the following fundamental result.
	\begin{thm}[\cite{Ser72}]
		If $E$ is an elliptic curve over $\mathbb{Q}$ without CM, then the associated adelic Galois representation
		$\rho_{E}: \Gal(\bar{\Q}/\Q) \ra \varprojlim_{m}\ \Aut(E[m]) \cong \GL_{2}(\widehat{\Z})$
		has open image. In particular, $[\GL_{2}(\widehat{\Z}): \Ima(\rho_{E})] < \infty.$ 
	\end{thm}
	Analogously, Pink and Rütsche~\cite{PR09} studied the adelic Galois representations attached to Drinfeld $A$-modules, the function field analogues of elliptic curves, and proved the following result.
	\begin{thm}[\cite{PR09}]
		Let $\varphi$ be a generic Drinfeld $A$-module of rank $r$. If $\End_{\bar{F}}(\varphi)=\varphi(A)$, then the image of the associated adelic Galois representation $\rho_{\varphi}(G_{F})$ is open in $\GL_{r}(\widehat{A})$. Equivalently,  $[\GL_{r}(\widehat{A}):\rho_{\varphi}(G_{F})]<\infty.$
	\end{thm}
	A natural question in both settings is to determine when the image has index $1$, i.e., when the adelic Galois representation is surjective? For elliptic curves over $\Q$, Serre proved in~\cite{Ser72} that the image has index at least $2$ in $\GL_{2}(\widehat{\Z})$. In contrast, Greicius showed in~\cite{Gre10} that there exists a number field $K$ and an elliptic curve $E/K$ for which the adelic Galois representation is surjective.

    However, for Drinfeld $A$-modules, there exist many explicit examples admitting surjective adelic Galois representations (cf.~ \cite{Hay74, Zyw11, Zyw25, Che22, Che25, KS25, KS26}). In particular, this implies that their $\mfp$-adic Galois representations are surjective for every $\mfp \in \Omega_A(:=\Spec(A)\setminus\{(0)\})$. In fact, in~\cite{Zyw25}, Zywina proved that there exists a density~$1$ subset of Drinfeld $A$-modules of rank~$2$ whose $\mathfrak{p}$-adic Galois representations are surjective for every $\mfp \in \Omega_A$.

    In a complementary direction, one can fix a prime ideal $\mfp \in \Omega_A$, and can look for Drinfeld $A$-modules having surjective $\mfp$-adic Galois representations. In~\cite{Ray24b}, Ray proved that, for every rank $r \ge 2$, the set of Drinfeld $A$-modules whose $(T)$-adic Galois representation is surjective has density~$1$. However, this result is existential in nature: it shows that almost all Drinfeld $A$-modules of rank $r \geq 2$ have surjective $(T)$-adic Galois representations, but it does not furnish an explicit criterion to decide, for a given Drinfeld $A$-module of rank $r \ge 2$, whether its associated $(T)$-adic Galois representation is surjective.

    However, especially for rank~$2$, in~\cite{Ray24a}, Ray gave an explicit criterion on the coefficients of $\varphi_T$, formulated in terms of certain valuation conditions, that guarantees the surjectivity of the associated $(T)$-adic Galois representation. Moreover, he proved that such Drinfeld modules
    have lower density at least $\frac{1}{q^7}$. More precisely:
    \begin{thm}[\cite{Ray24a}, Theorem 3.1]
            \label{Ray24a_Main_Result_1}
			Let $q\geq 5$ be odd and $\eta\in \F_{q}^{\times}$ be a non-square, and let $a_{1}, a_{2}$ be distinct non-zero elements in $\F_{q}^{\times}$. Let $\varphi$ be a Drinfeld $A$-module of rank $2$ defined by $\varphi_{T}=T+g_{1}\tau+g_{2}\tau^2$ where $g_{1},g_{2}\in A$ such that
			\begin{itemize}
				\item $\nu_{(T)}(g_1)=0$, $\nu_{(T-a_{1})}(g_{1})\geq 1$, and $\nu_{(T-a_2)}(g_{1})=0$,
				\item $\nu_{(T)}(g_2)=0$, $g_{2}\equiv -a_{1}\eta^{-1}\pmod{(T-a_{1})}$ and $\nu_{(T-a_{2})}(g_{2})=1$.
			\end{itemize} 
			Then, the representation $\rho_{\varphi,(T)}$ is surjective. Moreover, the set of such Drinfeld $A$-module of rank $2$ has lower density at least $ q^{-7}$. 
            \end{thm}
    In this article, for Drinfeld $A$-modules of ranks $2$, $3$, we give explicit verifiable conditions that ensure the surjectivity of the associated $(T)$-adic Galois representation, while improving upon existing criteria as in  Theorem~\ref{Ray24a_Main_Result_1} for rank $2$.
    
    For rank  $r \geq 4$, obtaining analogous results becomes technically more involved, making the situation considerably more delicate. Even in the case $r=3$, although the explicit verfiable conditions on the coefficients of  $\varphi_T$  resemble those in the rank $2$ case, the proof requires a different approach. Moreover, this method does not seem to extend readily to ranks $r \geq 4$.

    To state our main Theorems, we first introduce the following notation. 
    For any integer $r\geq 2$, define $\mcW^{r}:=\{\vec{g}=(g_1,g_2,\ldots,g_r)\in A^r: g_r\neq 0\}$.
     \begin{dfn}
       Let $\mcS^{r}$ denotes the subset of $\mcW^{r}$  consisting of $\vec{g}=(g_1,g_2,\ldots,g_r)$ such that
       \begin{itemize}
           \item $\nu_{T}(g_i)=0$ for $i=1,2,\ldots,r$;
           \item there exists $\mfl\in \Omega_{A}\setminus{(T)}$ such that $\nu_{\mfl}(g_{r-1})=0$ and $p\nmid\nu_{\mfl}(g_r)$.
       \end{itemize}
   \end{dfn}

  We now state the main Theorem of this article.
 \begin{thm*}[Theorems~\ref{T_adic_sur_rank_3}, \ref{density_rank_3}, \ref{T_adic_sur_rank_2}, \ref{density_rank_2}]
For $\vec g=(g_1,g_2,\ldots,g_r)\in \mcW^r$, let  $\varphi^{\vec g}$ be the corresponding Drinfeld $A$-module of rank $r$ defined by $$ \varphi^{\vec g}_T = T+g_1\tau+g_2\tau^2+\cdots+g_r\tau^r. $$

\begin{itemize}
\item[(i)] If $r=3$, $q>9$ odd, $\vec g\in \mcS^3$, then the $(T)$-adic Galois representation
$$
\rho_{\varphi^{\vec g},(T)}:G_F \longrightarrow
 \GL_3(A_{(T)})
$$
is surjective. Moreover, the density of $\mcS^3$ is $\left(1-\frac{1}{q}\right)^3.$

\item[(ii)] If $r=2$, $q\ge4$, and $\vec g\in \mcS^2$, then the $(T)$-adic Galois representation
$$
\rho_{\varphi^{\vec g},(T)}:G_F \longrightarrow
 \GL_2(A_{(T)})
$$
is surjective. Moreover, the density of $\mcS^2$ is  $\left(1-\frac{1}{q}\right)^2.$
\end{itemize}
\end{thm*}
A precise notion of density for a set of Drinfeld $A$-modules is given in 
\S $2.5$.

	\section{Preliminaries}
	Throughout this article, we stick to the following notations. 
    Let $q$ be a prime power and  $A:=\F_{q}[T]$ with the field of fractions $F:=\F_{q}(T)$, $G_{F}:=\Gal(F^{\sep}/F)$, where $F^{\sep}$ is the separable closure of $F$ in $\bar{F}$.  Let $\widehat{A} :=\varprojlim_{\mfa \lhd A}A/\mfa$, where $\mfa$ runs over all non-zero ideals of $A$, denote the profinite completion of $A$. For a commutative ring $R$ with unity, $R^{\times}$ denotes the set of all units in $R$.
    
	Let $\Omega_{A}:=\Spec(A)\setminus \{(0)\}$. To avoid notational complexity, we shall use $\mfa\subseteq A$ to denote both the generator and the ideal it generates. Let $\mfp\in \Omega_{A}$, $A_{\mfp}$ denote the completion of $A$ with respect to $\mfp$ with field of fraction $F_{\mfp}$, which is complete with respect to the normalized discrete valuation $\nu_p$, with the residue field $\F_{\mfp} := A_\mfp/\mfp A_\mfp(\cong A/\mfp)$. Throughout the article, we identify $\F_{(T)} $ with $ \F_q$, and use them interchangeably. Finally, $G_{F_{\mfp}}$ denotes the Galois group $\Gal(F_{\mfp}^{\sep}/F_{\mfp})$.
	
	\subsection{Drinfeld Modules}
    
    Let $K$ be an $A$-field $\gamma:A\ra K$. Let $K\{\tau\}$ denote the ring of twisted polynomials over $K$, i.e., polynomials in $\tau$ with coefficients in $K$ equipped with the usual addition and multiplication defined by $(a\tau^i)(b\tau^j):=ab^{q^i}\tau^{i+j}$. For $f(\tau)=a_{h}\tau^h+a_{h+1}\tau^{h+1}+\cdots+a_{d}\tau^d  \in K\{\tau\}$
	with $0\leq h\leq d$ and $a_{h} a_{d}\neq 0$, define
	$$\mathrm{ht}_{\tau}(f):=h\ \mathrm{and}\ \deg_{\tau}(f):=d.$$
	\begin{dfn}[Drinfeld module]
        \label{Drinfeld-Definition}
		A Drinfeld $K$-module of rank $r\geq 1$ is an $\F_{q}$-algebra homomorphism
		\[\begin{aligned}
			\varphi: A& \lra K\{\tau\} \\
			a & \longmapsto \varphi_a=\gamma(a)+g_1(a)\tau+\cdots+g_{r \deg_T(a)}(a)\tau^{r \deg_T(a)},
		\end{aligned}\]
		with $g_{r\deg_T(a)}(a)\neq 0$ if $a \neq 0$. If the coefficients $g_1(a),g_2(a),\ldots,g_{r \deg_T(a)}(a)\in K$ of $\varphi_{a}$ all lie in $A$ for all $a\in A$, then $\varphi$ is called a Drinfeld $A$-module. Also, $\varphi$ is referred to as generic if $\ker(\gamma)=0$.
    \end{dfn}
    \subsection{Torsion subgroups attached to  Drinfeld modules}
    For each $a\in A$, the element $\varphi_a$ corresponds to an $\F_q$-linear polynomial $\varphi_{a}(x)=\gamma(a)x+\sum_{i=1}^{r\deg_{T}(a)}g_i(a)x^{q^i}$. This endows $F^{\sep}$ with a (twisted) $A$-module structure via $b\cdot \alpha:=\varphi_b(\alpha)$ for all $b\in A$ and $\alpha\in F^{\sep}$, where $\varphi_b(\alpha)$ the evaluation of the $\F_q$-linear polynomial $\varphi_b(x)$ at $\alpha$.
    
    For any non-zero polynomial $a\in A$, the $a$-torsion of $\varphi$ is defined by
	$$\varphi[a]:=\{\alpha\in F^{\sep}|a\cdot \alpha=0 \}= \{\alpha\in F^{\sep}|\varphi_{a}(\alpha)=0 \}.$$
	Note that, $\varphi[a]$ is a finite dimensional $\F_{q}$-vector space. In fact, $\varphi[a]$ is an $A$-submodule of $F^{\sep}$ since for any
	$b\in A$ and $\alpha\in \varphi[a]$, we have
	$$\varphi_{a}(\varphi_{b}(\alpha))=\varphi_{ab}(\alpha)=\varphi_{ba}(\alpha)=\varphi_{b}(\varphi_{a}(\alpha))=0.$$
    For any non-zero ideal $\mathfrak{a}\subseteq A$ with monic generator $a$, we define $\varphi[\mfa]:= \varphi[a]$. By~\cite[Corollary 3.5.3]{Pap23}, the module $\varphi[\mfa]\cong (A/\mfa)^r$, if $\varphi$ is generic. 

    \subsection{Galois representations}
	Let $\varphi$ be a generic Drinfeld $F$-module of rank $r$. 
	\begin{itemize}
		\item The mod-$\mfa$  Galois representation of $\varphi$ is denoted by
		$$\bar{\rho}_{\varphi,\mfa}:G_{F}\longrightarrow \Aut_{A}(\varphi[\mfa]))\cong \GL_{r}(A/\mfa).$$
		
		\item For any $\mfp \in \Omega_A$, the $\mfp$-adic Galois representation of $\varphi$ is denoted by
		$${\rho}_{\varphi,\mfp}:G_{F}\longrightarrow \varprojlim_{i}\ \Aut_A(\varphi[\mfp^i])\cong\Aut_{A_{\mfp}}(T_{\mfp}(\varphi))\cong \GL_{r}(A_{\mfp}),$$ 
        where $T_{\mfp}(\varphi)$ denotes the $\mfp$-adic Tate module of $\varphi$. 
	\end{itemize}
    \subsection{Stable/Good reduction}
    
	For any Drinfeld $F$-module $\varphi$ and $\mfp\in \Omega_{A}$, the localized Drinfeld $F_\mfp$-module $\varphi_{\mfp}$ is defined by the composite
	$$\varphi_{\mfp}:A\xrightarrow{\varphi} F\{\tau\}\ra F_{\mfp}\{\tau\},$$
	where the second map arises from the inclusion $F\hookrightarrow F_{\mfp}$.
	
	\begin{dfn}[Stable/Good reduction]
		The Drinfeld $F$-module $\varphi$ of rank $r$ is said to have stable reduction at $\mfp$ if there exists a Drinfeld $A_{\mfp}$-module $\psi$, isomorphic to $\varphi_{\mfp}$ over $F_\mfp$, such that its reduction
		$$\bar{\psi}:A\lra \F_{\mfp}\{\tau\},$$
		is a Drinfeld $\F_{\mfp}$-module. The rank of $\bar{\psi}$ is called the reduction rank of $\varphi$ at $\mfp$. If the reduction rank of $\varphi$ at $\mfp$  is equal to $r$, then we say $\varphi$ has good reduction at $\mfp$.
	\end{dfn}
    

	\subsubsection{Height of a Drinfeld module}
	For $\mfp\in \Omega_{A}$ with good reduction for the Drinfeld $F$-module $\varphi$ of rank $r$ at $\mfp$, consider the reduced Drinfeld module
	$$\bar{\psi}:A\lra \F_{\mfp}\{\tau\}.$$
	Since $\mfp$ lies in the kernel of the reduction map $\gamma:A\ra \F_{\mfp}$, the constant term of $\bar{\psi}_{\mfp}$ is $0$. Hence, $\mathrm{ht}_{\tau}(\bar{\psi}_{\mfp})>0$. The height of $\bar{\psi}$ is defined by
	$$H_{\mfp}(\bar{\psi}):=\frac{\mathrm{ht}_{\tau}(\bar{\psi}_\mfp)}{\deg_{T}(\mfp)}.$$
    We refer to this as the reduction height of $\varphi$ at $\mfp$ and denote it by $H_\mfp$ 
    if $\varphi$ is clear from the context. By~\cite[Lemma 3.2.11]{Pap23}, the integer $H_{\mfp} \in [1, r]$.
	
	\subsection{Tate uniformization}
	\begin{dfn}
		Let $\mfp\in \Omega_{A}$ and $\psi:A\ra A_{\mfp}\{\tau\}$ be a Drinfeld $A_{\mfp}$-module. A $\psi$-lattice of rank $d$ is a free $A$-submodule $\Lambda\subseteq {}^{\psi}F_{\mfp}^{\sep}$ of rank $d$, which is invariant under the action of $G_{F_{\mfp}}$ and discrete with respect to the topology of the local field  $F_{\mfp}^{\sep}$. Here, ${}^{\psi}F_{\mfp}^{\sep}$ denotes the $A$-action on $F_{\mfp}^{\sep}$ induced via $\psi$.
	\end{dfn}
	
	Let $r$ and $d$ be two positive integers. A Tate datum over $A_{\mfp}$ is a pair $(\psi,\Lambda)$, where $\psi$ is a Drinfeld $A_{\mfp}$-module of rank $r$  and $\Lambda$ is a $\psi$-lattice of rank $d$. Two such Tate datum $(\psi,\Lambda)$ and $(\psi^\prime,\Lambda^\prime)$ are isomorphic if there is an isomorphism from $\psi$ to $\psi^\prime$ such that the induced $A$-module homomorphism ${}^{\psi}F_{\mfp}^{\sep}\ra {}^{\psi^\prime}F_{\mfp}^{\sep}$ gives an $A$-module isomorphism $\Lambda \ra \Lambda^\prime$.  The correspondence below is well-known, and we refer to it as the Drinfeld-Tate uniformization (cf.~\cite[Theorem 6.2.11]{Pap23}).
	
	\begin{thm}[\cite{Dri74}, \S 7]\label{Drinfeld}
		Let $r,d$ be two positive integers. There is a one-to-one correspondence between the following two sets:
		\begin{enumerate}
			\item The set of $F_{\mfp}$-isomorphism classes of Tate datum $(\psi,\Lambda)$ where $\psi$ is a Drinfeld $A_{\mfp}$-module of rank $r$ with good reduction, $\Lambda$ is a $\psi$-lattice of rank $d$.
			\item The set of $F_{\mfp}$-isomorphism classes of Drinfeld modules $\varphi:A\ra A_{\mfp}\{\tau\}$ of rank $r+d$ with stable reduction of rank $r$.
		\end{enumerate}
	\end{thm}
	In the proof of the above correspondence, one uses an entire $\F_{q}$-linear function denoted by $e_{\Lambda}(x)$, and defined by
	\begin{align}\label{e_Lambda_product_expression}
		e_{\Lambda}(x) := x \prod_{\substack{\lambda \in \Lambda \\ \lambda \ne 0}} \left(1 - \frac{x}{\lambda}\right) .
	\end{align}
	Here, we recall some important properties of $e_{\Lambda}(x)$:
	\begin{itemize}
		\item  The function $e_{\Lambda}(x)$ satisfies the relation $e_{\Lambda}    (\psi_{T}(x))=\varphi_{T}(e_{\Lambda}(x))$. 
        
		\item  The function $e_{\Lambda}(x)$ has a power series expansion as
		\begin{align}\label{e_Lambda_power_series_expression}
			e_{\Lambda}(x)=u_{0}x+u_{1}x^{q}+u_{2}x^{q^2}+\cdots+u_{n}x^{q^n}+\cdots
		\end{align}
		where $u_{0}=1$, and $u_{n}\in \mfp A_{\mfp}$ for $n\geq 1$ with $\nu_{\mfp}(u_n)\to \infty$ as $n\to \infty$.
	\end{itemize}
    
	For any non-unit $a \in A \setminus \{0\} $, we have the following: 
	\begin{enumerate}
		\item There exists a short exact sequence of $A[G_{F_{\mfp}}]$-modules
		\begin{equation}
			\label{s_e_s_drinfeld_datum}
			0 \longrightarrow \psi[a] = \psi_a^{-1}(0) \longrightarrow \varphi[a] \xrightarrow{\psi_a} \Lambda/a\Lambda \longrightarrow 0.
		\end{equation}
		\item There exists an $A[G_{F_{\mfp}}]$-module isomorphism
		\begin{align}
			\label{iso_drinfeld_datum}
			e_{\Lambda}: \psi_{a}^{-1}(\Lambda)/\Lambda  &\xrightarrow{\sim} \varphi[a]\\
			z+\Lambda &\mapsto e_{\Lambda}(z).\notag
		\end{align}
		\item We have
		\begin{align}\label{expression_of_varphi_mfl_in_terms_of_e_Lambda}
			\varphi_{a}(x)=ax\cdot\prod_{\substack{0\neq \gamma^\prime \in \psi_{a}^{-1}(\Lambda)/\Lambda}}\left(1 - \frac{x}{e_{\Lambda}(\gamma^\prime)} \right).
		\end{align}
	\end{enumerate}
	
	\subsection{Density of subsets of $A^r$}
For $\vec{g}\in \mcW^{r}$, $r\geq 2$,   let $\varphi^{\vec{g}}$ be the Drinfeld $A$-module of rank $r$ 
corresponding to $\vec{g}$, defined by
	$$\varphi^{\vec{g}}_{T}=T+g_{1}\tau+g_{2}\tau^2+\cdots+g_r\tau^r.$$ 
	For any polynomial $g\in A$, define $|g|:=q^{\deg_{T}(g)}$ if $g\neq 0$, and $|g|:=0$ otherwise. Define $|\vec{g}|:=\max\{|g_1|,|g_2|,\ldots,|g_r|\}$. For an integer $N>0$, consider
	$$\mcW^{r}(N):=\{\vec{g}\in\mcW^{r} \;\mid\; |\vec{g}|<q^N\},$$ 
	i.e., the collection of all elements of $\vec{g}\in\mcW^{r}$ such that $\deg_{T}(g_i)<N$ for all $i=1,2,\ldots,r.$ Therefore, $|\mcW^{r}(N)|=q^{(r-1)N}(q^{N}-1)$. 
    
    For any subset $\mcS^{r}\subseteq \mcW^{r}$, consider
	$$\mcS^{r}(N):=\mcS^{r}\cap \mcW^{r}(N)=\{\vec{g}\in\mcS^{r}\;\mid\;|\vec{g}|<q^N\}.$$
	The density of $\mcS^{r}$, denoted by $\mathfrak{d}(\mcS^{r})$, is defined by	
	$$\mathfrak{d}(\mcS^{r}):= \lim_{N \to \infty} \frac{|\mcS^{r}(N)|}{|\mcW^{r}(N)|},$$
    provided the limit exists. The upper density (resp., lower density) of $\mcS^{r}$,
    denoted by $\bar{\mathfrak{d}}(\mcS^{r})$ (resp., $\underline{\mathfrak{d}}(\mcS^{r})$), 
    is defined by
    \begin{align*}
        \bar{\mathfrak{d}}(\mcS^{r}):= \limsup_{N \to \infty}\frac{|\mcS^{r}(N)|}{|\mcW^{r}(N)|},  \left( \mathrm{resp.,}\quad  \underline{\mathfrak{d}}(\mcS^{r}):= \liminf_{N \to \infty}\frac{|\mcS^{r}(N)|}{|\mcW^{r}(N)|}\right).
    \end{align*}
	In contrast to the density $\mathfrak{d}(\mcS^{r})$, which need not converge, the quantities $\bar{\mathfrak{d}}(\mcS^{r})$ and $\underline{\mathfrak{d}}(\mcS^{r})$ always exist. 

	\section{Drinfeld $A$-modules of rank $3$ with surjective $(T)$-adic Galois representations}
    For $\vec g=(g_1,g_2,g_3)\in \mcW^{3}$, let $\varphi^{\vec g}_T = T+g_1\tau+g_2\tau^2+g_3\tau^3$ be the corresponding Drinfeld $A$-modules of rank $3$.
    In this section, we provide explicit sufficient conditions on the coefficients of $\varphi^{\vec g}_T$ under which the associated $(T)$-adic representation $\rho_{\varphi^{\vec g},(T)}$ is surjective. We show such Drinfeld $A$-modules have density $\big(1-\frac{1}{q}\big)^3.$
    
    Let $\mcS^{3}$ denotes the subset of $\mcW^{3}$ consisting of $\vec{g}=(g_1,g_2,g_3)$ such that
\begin{itemize}
\item $\nu_T(g_i)=0$ for $i=1,2,3$, and
\item there exists $\mfl\in\Omega_A\setminus (T)$ such that $\nu_{\mfl}(g_2)=0$ and $p\nmid\nu_{\mfl}(g_3)$.
\end{itemize}
The main theorem of this section is as follows.
\begin{thm}
   \label{T_adic_sur_rank_3}
   Let $q> 9$ be odd. For $\vec{g}\in \mcS^{3}$, let $\varphi^{\vec{g}}$ be the  Drinfeld $A$-module defined by $ \varphi^{\vec{g}}_{T}=T+g_1\tau+g_2\tau^2+g_3\tau^3.$ Then, the $(T)$-adic Galois representation  ${\rho}_{\varphi^{\vec{g}},(T)}:G_{F}\longrightarrow \GL_{3}(A_{(T)})$     is surjective.
       \end{thm}
       The proof of Theorem~\ref{T_adic_sur_rank_3} is based on a proposition by Pink and R\"utsche (cf. \cite[Proposition 4.1]{PR09}), which we recall now. 
       \begin{prop}
       \label{Pink_R_cond_T_adic_sur}
       Let $ r \in \N,\ \mfp\in \Omega_A$. Let $M$ be a closed subgroup of $\GL_{r}(A_{\mfp})$ such that $\det(M)=A_{\mfp}^{\times}$. Assume that
		$|\F_{\mfp}|\geq 4$. Suppose $M\equiv\GL_{r}(\F_{\mfp})\pmod \mfp$, and $M$ mod $\mfp^2$ contains a non-scalar matrix which is congruent to the identity modulo $\mfp$. Then, $M=\GL_{r}(A_{\mfp})$.
	\end{prop}
    To prove Theorem~\ref{T_adic_sur_rank_3}, it is enough to show that $M:=\Ima(\rho_{\varphi^{\vec{g}},(T)})$ satisfies the hypotheses of Proposition~\ref{Pink_R_cond_T_adic_sur}. As a first step, we establish the following.
    \begin{prop}\label{det_of_T_adic_rep_sujective_3}
           For $\vec{g} \in \mcS^3$, the determinant map
           $$\det\rho_{\varphi^{\vec{g}},(T)}:G_{F}\ra A_{(T)}^{\times}$$
           is surjective, i.e., $\det(M)=A_{(T)}^{\times}$.
    \end{prop}
    \begin{proof}
Since $\vec{g} \in \mcS^3$, we get $\varphi^{\vec{g}}$ has a good reduction at $(T)$ and the reduction height $H_{(T)}(\overline{\varphi}^{\vec{g}})$ of $\varphi^{\vec{g}}$ at $(T)$ is $1$. 
By~\cite[Proposition 2.3]{Ray24a}, the claim follows.
    \end{proof}
    Since $\vec{g}\in \mcS^{3}$, there exists $\mfl\in\Omega_A\setminus (T)$ such that $\nu_{\mfl}(g_2)=0$ and $p\nmid\nu_{\mfl}(g_3)$. To verify that $M$ satisfies the remaining hypotheses of Proposition~\ref{Pink_R_cond_T_adic_sur}, we first try to analyze the subgroup $\bar{\rho}_{\varphi^{\vec{g}}, (T)}(I_{\mfl})$ of $\overline{M}:=\Ima(\bar{\rho}_{\varphi^{\vec{g}}, (T)})\subseteq\GL_{3}(A/(T))$, where $I_{\mfl}$ denotes the inertia subgroup of $G_F$ at $\mfl$. 
    
\begin{prop}
\label{rank_3_A_nice_subgroup_of_bar_rho_varphi_vec_g_T_of_inertia_at_mfl}
For $\vec{g} \in \mcS^3$, there is a basis of $\varphi^{\vec{g}}[T]$ such that 
        $$\left\{\begin{pmatrix}
			1 & 0 & b_{\sigma,1}\\
			0 & 1 & b_{\sigma,2}\\
			0 & 0 & 1 \end{pmatrix} \; \middle| \; b_{\sigma,1}, b_{\sigma,2}\in A/(T)\right\}\subseteq \bar{\rho}_{\varphi^{\vec{g}},(T)}(I_{\mfl}).$$
        In particular, $|A/(T)|^2=q^{2}$ divides $|\overline{M}|$.
\end{prop}
	To prove Proposition~\ref{rank_3_A_nice_subgroup_of_bar_rho_varphi_vec_g_T_of_inertia_at_mfl}, we first establish the following Lemma.
	\begin{lem}\label{Im_of_inertia_l_noteq_T}
		For $\vec{g} \in \mcS^3$, there is a basis of $\varphi^{\vec{g}}[T]$ such that $\bar{\rho}_{\varphi^{\vec{g}},(T)}(I_{\mfl})$ is contained in the set $\left\{ \psmat{I_{2}}{*}{0}{c} : c\in  \F_{q}^{\times}\right\}$, $I_{2}$ denotes a $2 \times 2$ identity matrix.
	\end{lem}
	\begin{proof}
		Recall that, the Drinfeld module $\varphi^{\vec{g}}_{T}=T+g_{1}\tau+g_{2}\tau^2+g_{3}\tau^{3}$ has stable reduction at $\mfl$ of rank $2$, because $\mfl\mid g_{3}$ and $\mfl\nmid g_{2}$. By Theorem~\ref{Drinfeld}, 
		the corresponding Drinfeld datum $(\psi,\Lambda)$, where $\psi:A\ra A_{\mfl}\{\tau\}$ is a Drinfeld $A_{\mfl}$-module of rank $2$ with good reduction at $\mfl$ and $\Lambda$ is a $\psi$-lattice of rank $1$. By \eqref{iso_drinfeld_datum}, it is enough to find a basis of $\psi_{T}^{-1}(\Lambda)/\Lambda$ such that the action of $I_{\mfl}$ on $\psi_{T}^{-1}(\Lambda)/\Lambda$ is of the form 
        $\left\{ \psmat{I_{2}}{*}{0}{c} : c\in  \F_{q}^{\times}\right\}$.

		The Drinfeld $A_{\mfl}$-module $\psi$ is of rank $2$ and generic, by~\cite[Corollary 3.5.3]{Pap23}, there is a $A/(T)$-basis $\{w_{1},w_{2}\}$ of $\psi[T]$. Since $\psi$ 
		has good reduction at $\mfl$ and $\mfl\neq (T)$, the Galois representation $\bar{\rho}_{\psi,(T)}: G_{F_{\mfl}}\ra\Aut(\psi[T])$ is unramified at $\mfl$. In particular, $\sigma(w_{i})=w_{i}$ for all $\sigma\in I_{\mfl}$ and for $i=1,2$.   
		
		Since $\Lambda$ is a free $A$-module of rank $1$, we may fix a generator $\lambda$ of $\Lambda$. We can choose $z\in F_{\mfl}^{\sep}$ such that $\psi_{T}(z)= \lambda$. Since $\Lambda$ is stable under the Galois action of $G_{F_{\mfl}}$, there is a character
		$\chi_{\Lambda}:G_{F_{\mfl}}\ra \Aut(A)=A^{\times}=\F_{q}^{\times}$ such that $\sigma(\lambda)=\chi_{\Lambda}(\sigma)\lambda$ for all $\sigma\in G_{F_{\mfl}}$. Since $\psi_{T}$ is compatible with the action of $G_{F_{\mfl}}$, by~\eqref{s_e_s_drinfeld_datum}, we have
		$$\psi_{T}(\sigma(z))=\sigma(\psi_{T}(z))=\sigma(\lambda)=\chi_{\Lambda}(\sigma)\lambda=\chi_{\Lambda}(\sigma)\psi_{T}(z)=\psi_{T}(\chi_{\Lambda}(\sigma)z).$$
		Thus, $\sigma(z)-\chi_{\Lambda}(\sigma)z\in \psi[T]$, therefore there are some elements $b_{\sigma,1},b_{\sigma,2},$ in $A/(T)$ such that
		$\sigma(z)-\chi_{\Lambda}(\sigma)z=b_{\sigma,1}w_{1}+b_{\sigma,2}w_{2}$, i.e., $\sigma(z)=b_{\sigma,1}w_{1}+b_{\sigma,2}w_{2}+\chi_{\Lambda}(\sigma)z$. Therefore, the action of $\sigma\in I_{\mfl}$ on $ \psi_{T}^{-1}(\Lambda)/\Lambda$ with respect to the basis $\{w_{1}+\Lambda, w_{2}+\Lambda,z+\Lambda\}$ is of the form 
		$$\csmat{1}{0}{b_{\sigma,1}}{0}{1}{b_{\sigma,2}}{0}{0}{\chi_{\Lambda}(\sigma)},\  \text{i.e.},\  \bar{\rho}_{\varphi^{\vec{g}},(T)}(I_{\mfl})\subseteq \left\{ \psmat{I_{2}}{*}{0}{c} : c\in  \F_{q}^{\times}\right\}.$$
    \end{proof}
	
	We now use some notations and facts from the proof of Lemma~\ref{Im_of_inertia_l_noteq_T}. Recall that 
	$(\psi, \Lambda)$ be the Tate datum corresponding to the Drinfeld module $\varphi^{\vec{g}}$. Note that, $\{w_1,w_2\}$ is the $A/(T)$-basis of $\psi[T]$
	and $\{e_{\Lambda}(w_1),e_{\Lambda}(w_2),e_{\Lambda}(z)\}$ is the $A/(T)$-basis of $\varphi^{\vec{g}}[T]$.
	Since $\mfl\neq (T)$ and $I_{\mfl}$ acts trivially on $w_1, w_2$,  we get $w_{1},w_2\in F_{\mfl}^{\un}$, the maximal unramified extension of $F_{\mfl}$, i.e., the fixed field of $I_{\mfl}$. In particular, $F_{\mfl}(w_1,w_2,z)\subseteq F_{\mfl}^{\un}(z)$.
	
	The valuation $\nu_{\mfl}$ can be extended uniquely to a valuation $\nu_{\mfl}^{\prime}$ on $F_{\mfl}(w_1,w_2,z)$, a finite extension of $F_{\mfl}$,
	and hence it is complete. The valuation $\nu_{\mfl}^{\prime}$ can be extended uniquely to a valuation $\nu_{\mfl}^{\prime\prime}$ on $F_{\mfl}^{\un}(z)$. To summarise, we have
	$$(F_{\mfl},\nu_{\mfl})\lra (F_{\mfl}(w_1,w_2,z), \nu_{\mfl}^{\prime})\lra (F_{\mfl}^{\un}(z), \nu_{\mfl}^{\prime\prime}).$$
    To prove Proposition~\ref{rank_3_A_nice_subgroup_of_bar_rho_varphi_vec_g_T_of_inertia_at_mfl}, we first establish the following lemma, whose proof follows by arguments similar to those in~\cite[Lemmas 4.2 and 4.3]{KS25}.

    \begin{lem}\label{valuation_of_e_Lambda}
        Let $\lambda,z,w_{1}, w_2$ be as in the proof of  Lemma~\ref{Im_of_inertia_l_noteq_T}. Then, we have the following.
		\begin{enumerate}
			\item $\nu_{\mfl}^{\prime}(e_{\Lambda}(z^{q^i}))=\nu_{\mfl}^{\prime}(z^{q^i})$ for $i<2$.
            \item Let $a_{1},a_{2},b \in A/(T)$, then
		\begin{align*}
			\nu_{\mfl}^{\prime}\left(e_{\Lambda}(a_1 w_1+a_{2}w_{2}+ b z)\right) =
			\begin{cases}
				\nu_{\mfl}^{\prime}(z) \quad \quad\quad\quad\quad\ \ \text{if } b \ne 0, \\
				\nu_{\mfl}^{\prime}(a_1 w_1 + a_{2} w_{2}) \quad \text{if } b = 0.
			\end{cases}
		\end{align*}
		\end{enumerate}
    \end{lem}
	The valuation $\nu_{\mfl}$ can also be extended uniquely to a valuation $\nu_{\mfl}^{\un}$ on $F_{\mfl}^{\un}$, which in turn can be extended uniquely to an valuation $\nu_{\mfl}^{e_{\Lambda}}$ on $F_{\mfl}^{\un}(e_{\Lambda}(z))$. Note that, we have $\nu_{\mfl}^{e_{\Lambda}}(F_{\mfl}^{\un}(e_{\Lambda}(z))^{\times})=\frac{1}{e[F_{\mfl}^{\un}(e_{\Lambda}(z)):F_{\mfl}^{\un}]}\Z$, where $e[F_{\mfl}^{\un}(e_{\Lambda}(z)):F_{\mfl}^{\un}]$ is the ramification index of $F_{\mfl}^{\un}(e_{\Lambda}(z))/F_{\mfl}^{\un}$. To summarise, we have
	$$(F_{\mfl}, \nu_{\mfl})\ra (F_{\mfl}^{\un}, \nu_{\mfl}^{\un}) \ra (F_{\mfl}^{\un}(e_{\Lambda}(z)), \nu_{\mfl}^{e_{\Lambda}})\ra (F_{\mfl}^{\un}(z), \nu_{\mfl}^{\prime\prime}).$$
	Now, we are in a position to prove Proposition~\ref{rank_3_A_nice_subgroup_of_bar_rho_varphi_vec_g_T_of_inertia_at_mfl}

	\begin{proof}[Proof of Proposition~\ref{rank_3_A_nice_subgroup_of_bar_rho_varphi_vec_g_T_of_inertia_at_mfl}]
		Recall that
		\begin{align}\label{image_of_inertia_at_T_via_mod_mfl_rep}
			\bar{\rho}_{\varphi^{\vec{g}},(T)}(I_{\mfl})\cong I_{\mfl}/\Gal(F_{\mfl}^{\sep}/F_{\mfl}^{\un}(\varphi^{\vec{g}}[T]))\cong \Gal(F_{\mfl}^{\un}(\varphi^{\vec{g}}[T])/F_{\mfl}^{\un}),
		\end{align}
		where $I_{\mfl} :=\Gal(F^{\sep}_{\mfl}/F_{\mfl}^{\un})$. Since $w_{1},w_{2} \in F_{\mfl}^{\un}$,  we have $F_{\mfl}^{\un}(\varphi^{\vec{g}}[T])=F_{\mfl}^{\un}(e_{\Lambda}(z))$.
		By~\eqref{expression_of_varphi_mfl_in_terms_of_e_Lambda}, we have
		$$\varphi^{\vec{g}}_{T}(x)=T x\cdot\prod_{\substack{0\neq \gamma\prime \in\psi_{T}^{-1}(\Lambda)/\Lambda}}\left(1 - \frac{x}{e_{\Lambda}(\gamma\prime)} \right).$$
		Now, by comparing the leading coefficient on both sides of the above equation, up to units of $A_{\mfl}$, we  get
		$$ g_{3}
		= \pm  T
		\prod\limits_{\substack{0 \ne \gamma' \in \psi_{T}^{-1}(\Lambda)/\Lambda}}
		e_\Lambda(\gamma')^{-1}.$$
		By taking $\nu_{\mfl}^{\prime}$ on both sides and using the fact that $\{w_{1}+\Lambda,w_{2}+\Lambda,z+\Lambda\}$ is a $A/(T)$-basis of $\psi^{-1}_{T}(\Lambda)/\Lambda$, we get
		$$\nu_{\mfl}^{\prime}(g_3) = -\displaystyle\sum_{\substack{a_{1},a_{2},b\in A/(T)\\ \text{not all zero}}}^{}\nu_{\mfl}^{\prime}\big(e_{\Lambda}(a_{1}w_{1}+a_{2}w_{2}+bz)\big).$$
		By Lemma~\ref{valuation_of_e_Lambda}(2), we have
		$$ \nu_{\mfl}^{\prime}(g_3)=-q^{2}(q-1)\nu_{\mfl}^{\prime}(z) -\nu_{\mfl}^{\prime}(s),$$
		where $s=\prod\limits_{\substack{a_{1},a_{2}\in A/(T)\\ \text{not all zero}}}(a_{1}w_{1}+a_{2}w_{2})$, which is the product of all the roots of $\psi_{T}(x)/x$ as well as the constant term $T$ of $\psi_{T}(x)/x$. Since $\mfl\neq (T)$, we get
		$$\nu_{\mfl}^{\prime}(z)=-\frac {\nu_{\mfl}^{\prime}(g_3)}{q^{2}(q-1)}=-\frac {\nu_{\mfl}(g_3)}{q^{2}(q-1)}.$$
		Since $\nu_{\mfl}^{e_{\Lambda}}(F_{\mfl}^{\un}(e_\Lambda(z))^{\times})=\frac{1}{e[F_{\mfl}^{\un}(e_\Lambda(z)):F_{\mfl}^{\un}]}\Z$ and $\nu_{\mfl}^{e_{\Lambda}}(e_{\Lambda}(z))=\nu_{\mfl}^{\prime\prime}(e_{\Lambda}(z))=\nu_{\mfl}^{\prime}(e_{\Lambda}(z))=\nu_{\mfl}^{\prime}(z)$, where the last equality holds  by Lemma~\ref{valuation_of_e_Lambda}(1), hence, the number $q^2$ divides $e[F_{\mfl}^{\un}(e_\Lambda(z)):F_{\mfl}^{\un}]$, because $p\nmid\nu_{\mfl}(g_3)$. By~\eqref{image_of_inertia_at_T_via_mod_mfl_rep}, we have $|\bar{\rho}_{\varphi^{\vec{g}},(T)}(I_{\mfl})|\geq q^2$. 
        
        Now, we need to consider two cases to prove the claim.
        \begin{itemize}
            \item[(Case 1:)] Assume $|\bar{\rho}_{\varphi^{\vec{g}},(T)}(I_{\mfl})|=q^2$. By Lemma~\ref{Im_of_inertia_l_noteq_T}, let $B=\csmat{1}{0}{b_{\sigma,1}}{0}{1}{b_{\sigma,2}}{0}{0}{b_{\sigma,3}}\in \bar{\rho}_{\varphi^{\vec{g}},(T)}(I_{\mfl})$ be an element. Note that $b_{\sigma,3}\in \F_{q}^{\times}$ with order $t$. Since $B^{q^2}=I_{3\times3}$, we have $b_{\sigma,3}^{q^2}=1$. Therefore, $t\mid \gcd((q-1),q^2)=1$, and hence $b_{\sigma,3}=1$.
            
            \item[(Case 2:)] Assume $|\bar{\rho}_{\varphi^{\vec{g}},(T)}(I_{\mfl})|>q^2$. By Lemma~\ref{Im_of_inertia_l_noteq_T}, we get        
            $|\bar{\rho}_{\varphi^{\vec{g}}, (T)}(I_\mfl)|=q^2s$ where $1<s|(q-1)$. By Sylow Theorem, there exists $p$-Sylow subgroup of $\bar{\rho}_{\varphi^{\vec{g}}, (T)}(I_\mfl)$ of order $q^2$. Now, by Case 1, we are done. 
        \end{itemize}
    This completes the proof of the Proposition~\ref{rank_3_A_nice_subgroup_of_bar_rho_varphi_vec_g_T_of_inertia_at_mfl}.
\end{proof}
We now show that $M$ satisfies the second hypothesis of Proposition~\ref{Pink_R_cond_T_adic_sur}. 

    \begin{prop}\label{non-scalar_mod_T_square}
        For $\vec{g} \in \mcS^3$, the image of $\bar{\rho}_{{\varphi^{\vec{g}},(T^2)}}$, i.e., $M$ mod $(T^2)$, contains a non-scalar matrix that is congruent to the identity modulo $(T)$.
    \end{prop}
    \begin{proof}
        The proof of Proposition~\ref{rank_3_A_nice_subgroup_of_bar_rho_varphi_vec_g_T_of_inertia_at_mfl} works for any non-zero ideal $\mfa\subseteq A$ that is prime to $\mfl$. Therefore, taking $\mfa=(T^2)$ gives 
        $$\left\{\csmat{1}{0}{b_{\sigma,1}}{0}{1}{b_{\sigma,2}}{0}{0}{1} \; \middle| \; b_{\sigma,1}, b_{\sigma,2}\in A/(T^2)\right\}\subseteq\bar{\rho}_{\varphi^{\vec{g}},(T^2)}(I_{\mfl}),$$
        which proves the claim.
    \end{proof}

    Therefore, it remains only to show the following.
    \begin{prop}\label{mod_T_surjectivity}
        For $\vec{g} \in \mcS^3$, the mod-$(T)$ Galois representation $\bar{\rho}_{{\varphi^{\vec{g}},(T)}}$ is surjective, i.e., $M\equiv\GL_{3}(\F_{(T)})\pmod{(T)}$.
    \end{prop}
    To prove Proposition~\ref{mod_T_surjectivity}, we first prove the irreducibility of the mod-$(T)$ Galois representation.
    \begin{lem}\label{varphi_T_irreducible_F_T_G_F_module}
		For $\vec{g} \in \mcS^3$, the $\F_{(T)}[G_{F}]$-module $\varphi^{\vec{g}}[T]$ is irreducible.
	\end{lem}
    \begin{proof}
        Let $F(\varphi^{\vec{g}}[T])$ denote the splitting field of the separable polynomial
        $\varphi^{\vec{g}}_{T}(x)/x$.        
        The claim follows if we can show that the action of $\Gal(F(\varphi^{\vec{g}}[T])/F)\subseteq G_{F}$ on $\varphi^{\vec{g}}[T]$ has only one Galois orbit, i.e., the action of $\Gal(F(\varphi^{\vec{g}}[T])/F)$ on the roots of $\varphi^{\vec{g}}_{T}(x)/x$ is transitive. Since $\varphi^{\vec{g}}_{T}(x)/x$ is separable, this action is transitive if and only if $\varphi^{\vec{g}}_{T}(x)/x$ is irreducible over $F$ if and only if its reciprocal polynomial 
        $$x^{q^3-1}\varphi^{\vec{g}}_{T}(x^{-1})/x^{-1}=Tx^{q^3-1}+g_1x^{q^3-q}+g_2x^{q^3-q^2}+g_3$$
        is irreducible over $F$. Since $T\in F^{\times}$, it is enough to show the polynomial
        $$x^{q^3-1}+\frac{g_1}{T}x^{q^3-q}+\frac{g_2}{T}x^{q^3-q^2}+\frac{g_3}{T}$$
        is irreducible over $F$. By the Eisenstein criterion, $x^{q^3-1}+\frac{g_1}{T}x^{q^3-q}+\frac{g_2}{T}x^{q^3-q^2}+\frac{g_3}{T}$ irreducible on $F_{(\frac{1}{T})}$, the completion of $F$ at $(\frac{1}{T})$, and hence irreducible  over $F$.
    \end{proof}

    Using Lemma~\ref{varphi_T_irreducible_F_T_G_F_module}, we are now in a position to prove Proposition~\ref{mod_T_surjectivity}.
     	\begin{proof}[Proof of Proposition~\ref{mod_T_surjectivity}] The claim is to show that the mod-$(T)$ Galois representation $\bar{\rho}_{{\varphi^{\vec{g}}(T)}}$ is surjective, i.e., $\overline{M}=\Ima(\bar{\rho}_{\varphi^{\vec{g}}, (T)}) =\GL_{3}(\F_{q})$. Suppose this is not true. Then, there exists a maximal subgroup $\widetilde{M}$ of $\GL_{3}(\F_{q})$ that contains  $\overline{M}$. By Proposition~\ref{det_of_T_adic_rep_sujective_3}, we also get $\det(\widetilde{M})=\F_{q}^{\times}$, which implies $[\widetilde{M}:\widetilde{M}\cap\SL_{3}(\F_{q})] = q-1$. In particular, we also get $\widetilde{M}\cap\SL_{3}(\F_{q})$ is a proper subgroup of $\SL_{3}(\F_{q})$. If not, then $\SL_{3}(\F_{q}) \subseteq \widetilde{M}$. Since $\det(\widetilde{M})=\F_{q}^{\times}$, we have $\widetilde{M}=\GL_{3}(\F_{q})$. This is a contradiction, since $\widetilde{M}$ is proper. 
			
			We now show that $\widetilde{M}$ contains the center $Z(\GL_{3}(\F_{q}))$. Suppose, if possible, $Z(\GL_{3}(\F_{q}))\nsubseteq \widetilde{M}$. Now, the maximality of $\widetilde{M}$ would imply that $\widetilde{M}Z(\GL_{3}(\F_{q}))=\GL_{3}(\F_{q})$. Since ${q}\nmid |Z(\GL_{3}(\F_{q}))|$ and $q^3\mid |\GL_{3}(\F_{q})|$, we get $q^3\mid |\widetilde{M}|$. Therefore, $\widetilde{M}$ contains a Sylow $p$-subgroup of $\GL_{3}(\F_{q})$. Since the action of $\overline{M}$ on $\F_{q}^3$ is irreducible by Lemma~\ref{varphi_T_irreducible_F_T_G_F_module}, so is the action of $\widetilde{M}$. Therefore, by \cite[Lemma 16]{Che22}, $M$ contains $\SL_{3}(\F_{q})$. Arguing as before, $\widetilde{M}=\GL_{3}(\F_{q})$, which is a contradiction to the fact that $\widetilde{M}$ is a proper subgroup.
			
			To summarise, we have the following information about $\widetilde{M}$:
			\begin{enumerate}
				\item $\widetilde{M}$ is a maximal subgroup of $\GL_{3}(\F_{q})$ that does not contain  $\SL_{3}(\F_{q})$,
				\item The action of $\widetilde{M}$ on $\F_{q}^3$, via the mod-$(T)$ representation $\bar{\rho}_{\varphi^{\vec{g}},(T)}$, is irreducible,
				\item By Proposition~\ref{rank_3_A_nice_subgroup_of_bar_rho_varphi_vec_g_T_of_inertia_at_mfl}, $q^2\mid |\widetilde{M}|$. Moreover, $q^2\mid |\widetilde{M}\cap \SL_{3}(\F_{q})|$ (as $|\widetilde{M}|=(q-1)|\widetilde{M}\cap \SL_{3}(\F_{q})|$),
				\item $Z(\GL_{3}(\F_{q}))\subseteq \widetilde{M}$.
			\end{enumerate}
			
			Recall that, the claim is to show $\overline{M}=\GL_{3}(\F_{q})$ and on the contrary, we assumed that $\overline{M}$ is proper and hence $\overline{M}$ is contained in a maximal subgroup $\widetilde{M}$ of $\GL_{3}(\F_{q})$. By Aschbacher's Theorem~\cite[Theorem 2.2.19]{BHR13}, the maximal subgroups of $\GL_{3}(\F_{q})$ which do not contain $\SL_{3}(\F_{q})$ are classified into $9$ classes: $8$ geometric classes $\mathcal{C}_{1},\ldots,\mathcal{C}_{8}$ and one special class $\mathcal{S}$. A brief description of these classes can be found in~\cite[Appendix A.2]{Che22}. To obtain the desired contradiction, it suffices to show that $\widetilde{M}$ cannot fall into these classes based on the information on $\widetilde{M}$.
			
			$\bullet \ \mathcal{C}_{1}$: Suppose that $\widetilde{M}$ belongs to $\mathcal{C}_{1}$, then
			$\widetilde{M}$ stabilizes a proper non-zero subspace of $\F_{q}^3$. This cannot happen since $\widetilde{M}$ acts irreducibly on $\F_{q}^3$.
			
			$\bullet \ \mathcal{C}_{2}$: Suppose that $\widetilde{M}$ belongs to $\mathcal{C}_{2}$, then, there is a direct sum decomposition of $\F_{q}^3$ into three $1$-dimensional subspaces. Then, 
			the action of $\widetilde{M}$ on $\F_{q}^3$ is of type $\GL_{1}(\F_{q})\wr S_{3}=\GL_{1}(\F_{q})^3\rtimes S_{3}$, the wreath product of $\GL_{1}(\F_{q})$ and the symmetric group $S_{3}$. So we have, $|\widetilde{M}|$ divide $|\GL_{1}(\F_{q})^3\rtimes S_{3}| = (q-1)^3\cdot 3!$. This is a contradiction, since 
			$q^2\mid|\widetilde{M}|$ and $q>9$. Therefore, $\widetilde{M}$ cannot 
			lie in $\mathcal{C}_{2}$.
			
			$\bullet \ \mathcal{C}_{3}$:  Suppose that $\widetilde{M}$ belongs to $\mathcal{C}_{3}$, then the action of $\widetilde{M}$  on $\F_{q}^3$ is of type $\GL_{1}(\F_{(q)^3})$. So $|\widetilde{M}|$ divides $|\GL_{1}(\F_{(q)^3})| = \left(q^3-1\right)$, which is a contradiction as $q^2\mid|\widetilde{M}|$.
			
			$\bullet \ \mathcal{C}_{4}$:  Since there is no integer between $1$ and $\sqrt{3}$, we need not consider this case.
			
			$\bullet \ \mathcal{C}_{5}$:  Suppose that $\widetilde{M}$ belongs to $\mathcal{C}_{5}$. Then, there is a  proper  subfield $\F_{q_{0}}$ of $\F_{q}$ such that a conjugate of $\widetilde{M}$ in $\GL_{3}(\F_{q})$ is a subgroup of $\langle Z(\GL_{3}(\F_{q})),\GL_{3}(\F_{q_{0}}) \rangle$. Note that $q=q_{0}^d$ where $d\geq 2$.    Therefore, $|\widetilde{M}|$ divides $|\langle Z(\GL_{3}(\F_{q})),\GL_{3}(\F_{q_{0}})\rangle|$. Now
			\begin{align*}
				|\langle Z(\GL_{3}(\F_{q})),\GL_{3}(\F_{q_{0}}) \rangle |&=\frac{|Z(\GL_{3}(\F_{q}))|\times|\GL_{3}(\F_{q_{0}})|}{|Z(\GL_{3}(\F_{q}))\cap\GL_{3}(\F_{q_{0}})|}\\
				&=\frac{(q-1)(q_{0}^3-1)(q_{0}^3-q_{0})(q_{0}^3-q_{0}^2)}{(q_{0}-1)}\\
				&=q_{0}^3(q-1)(q_{0}^3-1)(q_{0}^2-1)
			\end{align*}
			Since $q^{2}\mid |\widetilde{M}|$, we get $q^{2}=q_{0}^{2d} \mid q_{0}^3(q-1)(q_{0}^3-1)(q_{0}^2-1)$,  which is a contradiction as $2d>3$. Therefore, $\widetilde{M}$ cannot lie in $\mathcal{C}_{5}$.
			
			$\bullet \ \mathcal{C}_{6}$: Suppose that $\widetilde{M}$ belongs to $\mathcal{C}_{6}$. By~\cite[Page 114]{BHR13}, 
			we should have $q=p\equiv 1\pmod{3}$. Then there is an absolutely irreducible extraspecial $3$-group $E$ of order $3^{1+2}$ such that $E\trianglelefteq \widetilde{M}\leq N_{\GL_{n}(q)}(E)$, the normalizer of $E$ in $\GL_{n}(q)$, and the action of $\widetilde{M}$ on $\F_{q}^3$ is of type $3^{1+2}. \Sp_{2}(3)$ (cf. \cite[\S 1.2]{BHR13} for these notations). So we have $|\widetilde{M}|$ divides $|3^{1+2}. \Sp_{2}(3)|= 2^3\cdot 3^4$, which is a contradiction, as $q^2\mid|\widetilde{M}|$ and $q > 9$.
			
			
			$\bullet \ \mathcal{C}_{7}$:  Since there are no integers $t\geq 2$ and $m\geq 1$ such that $3=m^t$, we do not need to consider this case.
			
			$\bullet \ \mathcal{C}_{8}$: Suppose that $\widetilde{M}$ belongs to $\mathcal{C}_{8}$, then $\widetilde{M}$ preserves a non-degenerate classical form on $\F_{q}^3$ up to scalar multiplication. By classical form, we mean symplectic form, uniform form, or quadratic form:
			\begin{itemize}
				\item[(i)] Symplectic form: These forms exist only on even-dimensional vector spaces. In our case, the dimension is $3$, so they do not exist.
				
				\item[(ii)]  Uniform form: To have a unitary form on a vector space over a finite field $\F_{q}$, $q$ must be a square. So assume $q=(q^{\prime})^2$. Then, the action of $\widetilde{M}$ on $\F_{q}^3$ is of type $\GU_{3}(q^{\prime})$. Therefore, $|\widetilde{M}|$ divides $|\GU_{3}(q^{\prime})|$ and by~\cite[Theorem 1.6.22]{BHR13} $|\GU_{3}(q^{\prime})|=(q^{\prime})^3\left(q^{\prime}+1\right)\left((q^{\prime})^2-1\right)\left((q^{\prime})^3+1\right)$, which is a contradiction, as $q^2=(q^{\prime})^4$ divides $|\widetilde{M}|$.
				\item[(iii)]  Quadratic form: The action of $\widetilde{M}$ on $\F_{q}^3$ is of type $\GO_{3}(q)$. Therefore, $|\widetilde{M}|$ divides $|\GO_{3}(q)|$ and by~\cite[Theorem 1.6.22]{BHR13} $|\GO_{3}(q)|=2q(q^2-1)$, which is a contradiction, as $q^2$ divides $|\widetilde{M}|$ and $q > 9$.
			\end{itemize}
			Therefore, $\widetilde{M}$ cannot lie in $\mathcal{C}_{8}$.
			
			$\bullet \ \mathcal{S}$: For this special class, we need to look at the proper subgroup $\widetilde{M} \cap \SL_3(\F_{q})$ that contains $Z(\SL_{3}(\F_{q}))$ in $\SL_3(\F_{q})$. Note that we have the property $q^2\mid |\widetilde{M} \cap \SL_3(\F_{q})|$. Therefore, by~\cite[Theorem 4.10.2]{BHR13}, the group $\widetilde{M} \cap \SL_3(\F_{q})$ can be any of these subgroups: $\PSL_{3}(2)\times Z(\SL_{3}(\F_{q}))$, $3^{\cdot}\textrm{A}_{6}$, $3^{\cdot}\textrm{A}_6.2_3 $, $3^{\cdot}\textrm{A}_{7}$ (cf.~\cite[\S 1.2]{BHR13} for these notations). Observe that cardinality of $\PSL_{3}(\F_2)\times Z(\SL_{3}(\F_{q}))$ is either $2^3\cdot3\cdot7$ or $2^3\cdot3^2\cdot7$ depending on $|Z(\SL_{3}(\F_{q}))|=1$ or $3$. Now the cardinality of the remaining groups is $2^3\cdot3^3\cdot5$, $2^4\cdot3^3\cdot5$, $2^3\cdot3^3\cdot5\cdot7$, respectively. Since $q > 9$ and $q^2\mid |\widetilde{M} \cap \SL_3(\F_{q})|$, we get that
			$\widetilde{M} \cap \SL_3(\F_{q})$ cannot be any of them. Hence, $\widetilde{M}$ cannot lie in $\mathcal{S}$. 
            
        Therefore, we get $\overline{M}=\GL_{3}(\F_q)$, i.e., $M\equiv \GL_{3}(\F_{(T)})\pmod{(T)}$.
        \end{proof}

        Since $M=\Ima(\rho_{\varphi^{\vec{g}},(T)})$ satisfies the hypotheses of Proposition~\ref{Pink_R_cond_T_adic_sur}, we complete the proof of Theorem~\ref{T_adic_sur_rank_3}.
        We now calculate the density of $\mcS^{3}$.
        \begin{thm}\label{density_rank_3}
            Let $q>9$ be odd. The density $\mathfrak{d}(\mcS^{3})$ of $\mcS^{3}$ is
            $\big(1-\frac{1}{q}\big)^3.$
        \end{thm}
          \begin{proof}
           Define 
           $$\mcS^{3}_{1}:=\left\{\vec{g}=(g_1,g_2,g_3)\in \mcW^{3}\; \middle| \;\nu_{T}(g_i)=0\ \text{for}\ i=1,2,3\right\},\quad \text{and}$$ 
           $$\mcS^{3}_{2}:= \left\{\vec{g}=(g_1,g_2,g_3)\in \mcW^{3} \; \middle| \;
	\begin{array}{l}
		\text{there exists}\ \mfl\in \Omega_{A}\setminus{(T)} \text{ such that}  \\
		\nu_{\mfl}(g_2)=0\ \text{and}\ p\nmid\nu_{\mfl}(g_3)
	\end{array}
	\right\}.$$
          Clearly, $\mcS^{3}=\mcS^{3}_{1}\cap\mcS^{3}_{2}$. By~\cite[Proposition 6.1]{Ray24b}, we get $\mathfrak{d}(\mcS^{3})=\mathfrak{d}(\mcS^{3}_{1}\cap\mcS^{3}_{2})=\mathfrak{d}(\mcS^{3}_{1})$. Any polynomial $g=g(T)\in A$ with $\deg_{T}(g)<N$ and $T\mid g$ has to be of the form $g=Th$ where $h\in A$ with $\deg_{T}(h)<N-1$. Since the number of $g\in A$ with $\deg_{T}(g)<N$ and $T\mid g$ is $q^{N-1}$, the number of $g\in A$ with $\deg_{T}(g)<N$ and $T\nmid g$ is $q^N-q^{N-1}=q^N(1-\frac{1}{q}).$ Therefore, we have
          $$\mathfrak{d}(\mcS^{3})= \mathfrak{d}(\mcS^{3}_{1})=\lim_{N \to \infty} \frac{|\mcS^{3}_{1}(N)|}{|\mcW^{3}(N)|}=\lim_{N \to \infty}\frac{q^{3N}\big({1-\frac{1}{q}}\big)^3}{q^{2N}(q^N-1)}=
          \left(1-\frac{1}{q}\right)^3.$$
          \end{proof}
          We conclude this section with a remark highlighting comparisons between our results and those of in~\cite{Ray24b} for the rank $3$ case.
          
          \subsection{Comparisons with~\cite{Ray24b} for the rank $3$ case}
          In \cite{Ray24b}, Ray proves that Drinfeld $A$-modules for rank $r\geq 2$, in particular for $r=3$, with surjective $(T)$-adic Galois representation form a density $1$ subset. His argument relies on Hilbert irreducibility for function fields and yields an existential asymptotic result, but for a given triple $\vec{g}=(g_1,g_2,g_3)\in \mcW^{3}$, there is no explicit sufficient criterion to decide whether the corresponding Drinfeld $A$-modules for rank $3$ have surjective $(T)$-adic Galois representation.
          
          In this article, we have provided a sufficient, concrete and verifiable condition for rank $3$. We define an explicitly described subset $\mcS^{3}\subseteq \mcW^{3}$ using direct valuation conditions on the coordinates of $\vec{g}$, and prove that for every $\vec{g}\in\mcS^{3}$, the corresponding Drinfeld $A$-modules for rank $3$ has surjective $(T)$-adic Galois representation. Moreover, we compute the density of $\mcS^{3}$ as $(1-\frac{1}{q})^3$.
          

	\section{Drinfeld $A$-modules of rank $2$ with surjective $(T)$-adic Galois representations}

   For $\vec g=(g_1,g_2)\in \mcW^{2}$, let $\varphi^{\vec g}_T = T+g_1\tau+g_2\tau^2$ be the corresponding Drinfeld $A$-modules of rank $2$. In this section, we provide explicit sufficient conditions on the coefficients of $\varphi^{\vec g}_T$ under which the associated $(T)$-adic representation $\rho_{\varphi^{\vec g},(T)}$ is surjective. We show such Drinfeld $A$-modules have density $\big(1-\frac{1}{q}\big)^2.$
    
    Let $\mcS^{2}$ denotes the subset of $\mcW^{2}$ consisting of $\vec{g}=(g_1,g_2)$ such that
\begin{itemize}
\item $\nu_T(g_i)=0$ for $i=1,2$;
\item there exists $\mfl\in\Omega_A\setminus (T)$ such that $\nu_{\mfl}(g_1)=0$ and $p\nmid\nu_{\mfl}(g_2)$.
\end{itemize}
The main theorems of this section are as follows.
       \begin{thm}
       \label{T_adic_sur_rank_2}
           Let $q\geq 4$. For any $\vec{g}\in \mcS^{2}$, let $\varphi^{\vec{g}}$ be the corresponding Drinfeld $A$-module defined by $ \varphi^{\vec{g}}_{T}=T+g_1\tau+g_2\tau^2.$ Then, the $(T)$-adic Galois representation  ${\rho}_{\varphi^{\vec{g}},(T)}:G_{F}\longrightarrow \GL_{2}(A_{(T)})$     is surjective.
       \end{thm}

        As before, the proof of Theorem~\ref{T_adic_sur_rank_2} relies on Proposition~\ref{Pink_R_cond_T_adic_sur}, hence it is enough to show that $M:=\Ima(\rho_{\varphi^{\vec{g}},(T)})$ satisfies the hypotheses of Proposition~\ref{Pink_R_cond_T_adic_sur}.  At first, arguing as in Proposition~\ref{det_of_T_adic_rep_sujective_3}, we get
        \begin{prop}\label{det_of_T_adic_rep_sujective_2}
           For $\vec{g} \in \mcS^2$, the determinant map
           $$\det\rho_{\varphi^{\vec{g}},(T)}:G_{F}\ra A_{(T)}^{\times}$$
           is surjective, i.e., $\det(M)=A_{(T)}^{\times}$.
        \end{prop}

      Since $\vec{g}\in \mcS^{2}$, there exists $\mfl\in\Omega_A\setminus (T)$ such that $\nu_{\mfl}(g_1)=0$ and $p\nmid\nu_{\mfl}(g_2)$.  To verify that $M$ satisfies the remaining hypotheses of Proposition~\ref{Pink_R_cond_T_adic_sur}, we first try to analyze the subgroup $\bar{\rho}_{\varphi^{\vec{g}}, (T)}(I_{\mfl})$ of $\overline{M}:=\Ima(\bar{\rho}_{\varphi^{\vec{g}}, (T)})\subseteq\GL_{2}(A/(T))$.
      
     \begin{lem}\label{rank_2_A_nice_subgroup_of_bar_rho_varphi_vec_g_T_of_inertia_at_mfl}
		For $\vec{g} \in \mcS^2$, there exists a basis of $\varphi^{\vec{g}}[T]$ such that 
        $$
			\left\{ \begin{pmatrix} 1 & b_{\sigma,1} \\ 0 & 1 \end{pmatrix} \; \middle| \; b_{\sigma,1} \in A/(T)\right\}\subseteq\bar{\rho}_{\varphi^{\vec{g}}, (T)}(I_\mfl).
		$$
        In particular, $|A/(T)|=q$ divides $|\overline{M}|$.
	\end{lem}
    \begin{proof}
        The Drinfeld $A$-module $\varphi^{\vec{g}}$ has stable reduction of rank $1$ at $\mfl$, since $\mfl\mid g_2$ and $\mfl\nmid g_1$. Therefore, by  \cite[Proposition 4.1]{Zyw11}, there exists a basis of $\varphi^{\vec{g}}[T]$ such that  
        $$\bar{\rho}_{\varphi^{\vec{g}},(T)}(I_\mfl)\subseteq \left\{
			\psmat{1}{b_{\sigma,1}}{0}{b_{\sigma,2}} \; \middle| \; b_{\sigma,1} \in A/(T), b_{\sigma,2}\in \F_{q}^{\times}\right\}$$
        and $|\bar{\rho}_{\varphi^{\vec{g}}, (T)}(I_\mfl)|\geq e_{\varphi^{\vec{g}}}$ where $e_{\varphi^{\vec{g}}}$ is the order of $\frac{1}{(q-1)q}\nu_{\mfl}\Big(\frac{g_1^{q+1}}{g_2}\Big)+\mathbb{Z}$ in $\mathbb{Q}/\mathbb{Z}$. Since $\nu_{\mfl}(g_1)=0$,
        and $p\nmid \nu_{\mfl}(g_2)$, we get $e_{\varphi^{\vec{g}}}\geq q$, and hence
        $|\bar{\rho}_{\varphi^{\vec{g}}, (T)}(I_\mfl)|\geq q$. Now, arguing as in the proof of Proposition~\ref{rank_3_A_nice_subgroup_of_bar_rho_varphi_vec_g_T_of_inertia_at_mfl}, 
        the claim follows.
    \end{proof}

In~\cite[Proposition 4.1]{Zyw11}, the description of $\bar{\rho}_{\varphi,\mfa}(I_{\mfl})$ was stated for any arbitrary non-zero ideal $\mfa\subseteq A$. However, a closer inspection of the proof shows that~\cite[Proposition 4.1]{Zyw11} requires an additional assumption that $\mfl\nmid\mfa$. In particular, for $\mfa=(T^2)$, we have the following.


    \begin{prop}
    \label{non-scalar_mod_T_square_2}
        For $\vec{g} \in \mcS^2$, the image of $\bar{\rho}_{{\varphi^{\vec{g}},(T^2)}}$, i.e., $M$ mod $(T^2)$, contains a non-scalar element that is congruent to the identity modulo $(T)$.
    \end{prop}
    \begin{proof}
        The preceding argument shows that the proof of Lemma~\ref{rank_2_A_nice_subgroup_of_bar_rho_varphi_vec_g_T_of_inertia_at_mfl} applies to any non-zero ideal $\mfa\subseteq A$ that is prime to $\mfl$. Therefore, taking $\mfa=(T^2)$ gives
        $$ \left\{
			\psmat{1}{b_{\sigma,1}}{0}{1}\; \middle| \; b_{\sigma,1} \in A/(T^2)\right\}\subseteq\bar{\rho}_{\varphi^{\vec{g}}, (T^2)}(I_\mfl),$$
        which proves the claim.
    \end{proof}
    

    The only remaining hypothesis to verify in Proposition~\ref{Pink_R_cond_T_adic_sur} is that the representation $\bar{\rho}_{\varphi,T}$ is surjective, i.e., $M\equiv\GL_{2}(\F_{(T)})\pmod{T}$. By Proposition~\ref{det_of_T_adic_rep_sujective_2}, we get $\det(\overline{M})=\F_{(T)}^{\times}$. The required claim follows if we can show that $\SL_{2}(\F_{(T)})\subseteq \overline{M}$, which can be achieved by~\cite[Lemma A.1]{Zyw11}.

        \begin{lem}
        \label{zyw_sl_{2}_F}
		Let $\F$ be a finite field and $G$ be a subgroup of $\GL_2(\F)$ such that
		$G$ contains a Sylow $p$-subgroup of $\GL_2(\F)$, and $G$ acts irreducibly on $\F^2$, then $G$ contains $\SL_{2}(\F)$. 
     	\end{lem}
       By Lemma~\ref{rank_2_A_nice_subgroup_of_bar_rho_varphi_vec_g_T_of_inertia_at_mfl}, the group $\overline{M}$ contains the Sylow $p$-subgroup, the 
       full group of unipotent upper triangular matrices in $\GL_2(\F_{q})$. Now, arguing as in Lemma~\ref{varphi_T_irreducible_F_T_G_F_module},  $\overline{M}$ acts irreducibly on $\F_{q}^2 \cong \varphi^{\vec{g}}[T]$, i.e., the $\F_{(T)}[G_{F}]$-module $\varphi^{\vec{g}}[T]$ is irreducible. Hence, by Lemma~\ref{zyw_sl_{2}_F}, we get $\SL_{2}(\F_{q})\subseteq \overline{M}$. Therefore, we get
       $\overline{M} =\GL_{2}(\F_{q})$. Since $M=\Ima(\rho_{\varphi^{\vec{g}},(T)})$ satisfies the hypotheses of Proposition~\ref{Pink_R_cond_T_adic_sur},  we complete the  proof of Theorem~\ref{T_adic_sur_rank_2}.

       We now compute the density of $\mcS^{2}$. 

       \begin{thm}
       \label{density_rank_2}
            Assume that $q\geq 4$. The density $\mathfrak{d}(\mcS^{2})$ of $\mcS^{2}$ is $\big(1-\frac{1}{q}\big)^2.$
        \end{thm}
      \begin{proof}
           Define 
           $$\mcS^{2}_{1}:=\left\{\vec{g}=(g_1,g_2)\in \mcW^{2}\; \middle| \;\nu_{T}(g_i)=0\ \text{for}\ i=1,2\right\}, \quad \text{and}$$ 
           $$\mcS^{2}_{2}:= \left\{\vec{g}=(g_1,g_2)\in \mcW^{2} \; \middle| \;
	\begin{array}{l}
		\text{there exists}\ \mfl\in \Omega_{A}\setminus{(T)} \text{ such that}  \\
		\nu_{\mfl}(g_1)=0\ \text{and}\ p\nmid\nu_{\mfl}(g_2)
	\end{array}
	\right\}.$$
 Clearly, $\mcS^{2} =\mcS^{2}_{1}\cap\mcS^{2}_{2}$. By~\cite[Proposition 6.1]{Ray24b}, we get
 $\mathfrak{d}(\mcS^{2})=\mathfrak{d}(\mcS^{2}_{1}\cap\mcS^{2}_{2})=\mathfrak{d}(\mcS^{2}_{1})$. Now, arguing as in the proof of Theorem~\ref{density_rank_3}, 
 we get $\mathfrak{d}(\mcS^{2}_{1})=\big(1-\frac{1}{q}\big)^2.$
         \end{proof}

    Finally, we conclude this article with some remarks highlighting comparisons between our results and those of~\cite{Ray24a}, \cite{Zyw25}.

    \subsection{Comparison with \cite{Ray24a}}
       \begin{itemize}
       \item Here, we work with $\F_q$ with $q \ge 4$,  while~\cite{Ray24a} assumes  $q\geq 5$ is odd.

       \item We allow an arbitrary prime $\mfl\in\Omega_A\setminus (T)$ with $\nu_{\mfl}(g_1)=0$ and $p\nmid\nu_{\mfl}(g_2)$, where as \cite{Ray24a} works with a specific degree-$1$ prime $(T-a_2)$ and with stronger valuation conditions, namely $\nu_{(T-a_2)}(g_1)=0$ and $\nu_{(T-a_2)}(g_2)=1$.
       
       \item Both our work and~\cite{Ray24a} assume the conditions $\nu_T(g_i)=0$ for $i=1,2$, which implies that the reduction height of $\varphi^{\vec{g}}$ at $(T)$ is $1$. In our approach, these same conditions are also sufficient to prove the irreducibility of the mod-$(T)$ representation $\bar{\rho}_{\varphi^{\vec{g}},(T)}$, by directly showing that $\varphi^{\vec{g}}_{T}(x)/x$ is irreducible over $F$. In contrast,~\cite[Proposition 3.5]{Ray24a} establishes the corresponding result under additional hypotheses on the coefficients. 

       \item We compute the density of $\mcS^{2}$ as $(1-\frac{1}{q})^2$. In contrast,~\cite[Theorem 1.1]{Ray24a} computes the lower density.
       \end{itemize}

       
    \subsection{Comparison with \cite{Zyw25}}
    In the proof of~\cite[Theorem 9.1]{Zyw25}, Zywina constructs four subsets 
$\mathcal{R}$, $\mathcal{S}_m$, $\mathcal{T}_m$, and $\mathcal{U}_m$ of $A^2$ for $m\ge 2$ (see~\cite[p.~22]{Zyw25}), and proves that if $\vec{g} \in  \mathcal{R}\cap\mathcal{S}_m\cap\mathcal{T}_m\cap\mathcal{U}_m$, then the corresponding Drinfeld $A$-module $\varphi^{\vec{g}}$ has surjective $\mfp$-adic Galois representation  for all $\mfp\in\Omega_A$. Based on the definitions of these four sets, it appears that determining whether a given vector $\vec{g}  \in R \cap S_m \cap T_m \cap U_m$ cannot, in general, be characterized by a simple intrinsic valuation criterion depending solely on $(g_1, g_2)$.

 
 However, for $(T)$-adic surjectivity, we provide a concrete and purely coordinate-wise sufficient condition that guarantees $(T)$-adic surjectivity, and we have shown that such Drinfeld $A$-modules have density $(1-\frac{1}{q})^2$.
We conclude the article with a remark that our sufficient criteria yield additional examples of Drinfeld $A$-modules that were not covered in~\cite{Zyw25}. For example, 
\begin{itemize}
    \item  $\vec{g}=(1,\mfl)$ lies in $\mcS^{2} \setminus \mathcal{R}$,
    \item If  $\mathfrak p\in\Omega_A\setminus\{\mfl,(T)\}$ with $\deg_T(\mathfrak p)>m$, then $\vec{g}=(\mathfrak p,\mfl\mathfrak p) \in \mcS^{2} \setminus \mathcal{S}_m$.
\end{itemize}

	\bibliographystyle{plain, abbrv}

\end{document}